\documentclass[a4paper, 12pt,twoside]{amsart}
\usepackage[english]{babel}
\usepackage{amssymb, amsmath, eucal}
\usepackage[all]{xy}
\usepackage{mathrsfs}
\usepackage[dvips]{graphics}

\newtheorem{The}{Theorem}[section]
\newtheorem{Lem}[The]{Lemma}
\newtheorem{Prop}[The]{Proposition}
\newtheorem{Cor}[The]{Corollary}

\newtheorem{Ex}[The]{Example}

\newcommand{\C}{\mathbb{C}}
\newcommand{\R}{\mathbb{R}}
\newcommand{\N}{\mathbb{N}}
\newcommand{\Z}{\mathbb{Z}}

\begin{document}
 \title[ Maximal plurisubharmonic function]{Approximation of maximal plurisubharmonic functions}

\setcounter{tocdepth}{1}

  \author{Do Hoang Son} 
  
\address{Institute of Mathematics \\ Vietnam Academy of Science and Technology \\18
Hoang Quoc Viet \\Hanoi \\Vietnam}
\email{hoangson.do.vn@gmail.com }
\email{dhson@math.ac.vn}
 \date{\today\\ The author was funded by the Vietnam National Foundation for Science and Technology Development (NAFOSTED) under grant number 101.02-2017.306.}


\begin{abstract}
  Let $u$ be a maximal plurisubharmonic function in a domain $\Omega\subset\C^n$ ($n\geq 2$). It is classical that, for any $U\Subset\Omega$,
   there exists  a sequence of bounded plurisubharmonic functions $PSH(U)\ni u_j\searrow u$ satisfying the property: 
   $(dd^c u_j)^n$ is weakly convergent to $0$ as $j\rightarrow\infty$. In general, this property does not hold for arbitrary sequence. 
   In this paper, we show that \textit{for any} sequence of bounded plurisubharmonic functions $PSH(U)\ni u_j\searrow u$, 
   $(|u_j|+1)^{-a} (dd^cu_j)^n$ is  weakly convergent to $0$ as $j\rightarrow\infty$, where $a>n-1$.
   We also generalize some well-known results about approximation of maximal plurisubharmonic functions.
\end{abstract}

\maketitle

\tableofcontents
\newpage

\section*{Introduction}
 Let $\Omega\subset\C^n$ be a bounded domain $(n\geq 2)$. A function $u\in PSH(\Omega)$ is called maximal
 if for every open set $G\Subset\Omega$, and for each upper semicontinuous function $v$ on $\bar{G}$ such that
 $v\in PSH(G)$ and $v|_{\partial G}\leq u|_{\partial G}$, we have $v\leq u$. There are some equivalent descriptions of
 maximality which have been presented in \cite{Sad81}, \cite{Kli91}. We denote by
 	$MPSH(\Omega)$ the set of all maximal plurisubharmonic functions in $\Omega$.
 
   Sadullaev \cite{Sad81} has proved that,  if $u\in MPSH(\Omega)$ then for any $U\Subset\Omega$,
   there exists  a sequence of  functions $PSH(U)\cap C(U)\ni u_j\searrow u$ such that 
   $(dd^c u_j)^n$ is weakly convergent to $0$ as $j\rightarrow\infty$. 
   If $u$ belongs in the
   domain of definition  of Monge-Amp\`ere operator (see \cite{Ceg04},
   \cite{Blo06}), it implies that $(dd^cu)^n=0$. By the comparison principle \cite{BT82}, it is also a sufficient condition
   for maximality.
   
   In general, a maximal plurisubharmonic function may belong outside the
  domain of definition  of Monge-Amp\`ere operator, i.e.,  
  there exists also  a sequence of  functions $PSH(U)\cap C(U)\ni v_j\searrow u$ such that 
  $(dd^c v_j)^n$ is not weakly convergent to $0$ as $j\rightarrow\infty$. For example, Cegrell \cite{Ceg86} has shown that 
  if $u=\log |z_1|^2+...+\log |z_n|^2$ and $v_j=\log (|z_1|^2+\dfrac{1}{j})+...+\log (|z_n|^2+\dfrac{1}{j})$ then
  $u$ is maximal in $\C^n$ but $(dd^cv_j)^n$ is not convergent to $0$. Blocki \cite{Blo09} has given another example 
  that $u=-\sqrt{\log |z| \log |w|}$ is maximal in $\Delta^2\setminus\{(0, 0)\}$ but $(dd^c\max\{u, -j\})^n$ is not
  convergent to $0$.\\

 In this paper, we show that, for any sequence $u_j$, the sequence of
  weighted Monge-Amp\`ere operators $(|u_j|+1)^{-a}(dd^cu_j)^n$ 
 is convergent to $0$ for all $a>n-1$.
 \begin{The}\label{main1}
 	Let $u$ be a negative maximal plurisubharmonic function in $\Omega\subset\C^n$ and let $U, \tilde{U}$ be open subset of
 	$\Omega$ such that $U\Subset\tilde{U}\Subset\Omega$. Assume that $u_j\in PSH^-(\tilde{U})\cap L^{\infty}(\tilde{U})$ is
 	decreasing to $u$ in $\tilde{U}$. Then
 	\begin{equation}\label{main1.eq}
 	\int\limits_U(|u_j|+1)^{-a}(dd^cu_j)^n\stackrel{j\to\infty}{\longrightarrow}0,\,\forall a>n-1.
 	\end{equation}
 \end{The}
There is an interesting observation that if $u$ is maximal on every leaf of a holomorphic foliation on $\Omega$ then $u$ is maximal in 
$\Omega$. We generalize this observation as following:
  \begin{The}\label{main2}
  	Let $\Omega\subset\C^n$ be a domain and $1\leq q<n$. Let $\{T_j\}_{j\in J}$ be a family of
  	closed, positive currents of bidegree $(n-q, n-q)$ in $\Omega$ such that
  	\begin{center}
  		$\int\limits_{\Omega}T_j\wedge\omega^q>0,$
  	\end{center}
  	for all $j\in J$, where $\omega=dd^c|z|^2$. 
  	Assume also that, for any Borel set $E\subset\Omega$, if $\int\limits_ET_j\wedge\omega^q=0$ for all $j\in J$
  	then $\lambda (E)=0$, where $\lambda$ is Lebesgue measure. Let $u\in PSH(\Omega)$ and, for any $k\in\Z^+$, $j\in J$,
  	let $u_{jk}\in PSH(\Omega)\cap L^{\infty}(\Omega)$
  	such that $u_{jk}$ is decreasing to $u$ as $k\rightarrow\infty$. If
  	\begin{center}
  		$\lim\limits_{k\to\infty}\int\limits_{\Omega}T_j\wedge (dd^cu_{jk})^q=0,$
  	\end{center}
  	for all $j\in J$ then $u$ is maximal.
  \end{The} 
  In \cite{Blo09}, Blocki has proved that 
   $u=-\sqrt{\log |z| \log |w|}$ is maximal in $\Delta^2\setminus\{(0, 0)\}$ but $(dd^c\max\{u, -j\})^n$ is not
  convergent to $0$. By using another method of calculation, we improve Blocki's result as the following:
  \begin{The}\label{main3}
  Let $\chi_m : (-\infty, 0)\rightarrow (-\infty, 0)$ be bounded convex non-decreasing functions such that $\chi_m$ is decreasing
  to $Id$ as $m\rightarrow\infty$. Let $u=-\sqrt{\log |z| \log |w|}$ and $u_m=\chi_m (u)$. Then $(dd^c u_m)^n$ is not weakly
  convergent to $0$ as $m\rightarrow\infty$.
  \end{The}
   \section{On the weighted Monge-Amp\`ere operator}
   \subsection{A class of maximal plurisubharmonic functions}
    We say that a function $u\in PSH^{-}(\Omega)$ has $M_1$ property iff for 
    every open set $U\Subset\Omega$, there are $u_j\in PSH^{-}(U)\cap C(U)$ such that $u_j$ is decreasing to $u$ in $U$ and
    \begin{equation}\label{M1prop.eq}
    \lim\limits_{j\to\infty}\left(\int\limits_{U\cap\{u_j>-t\}}(dd^c u_j)^n
    +\int\limits_{U\cap\{u_j>-t\}}du_j\wedge d^cu_j\wedge(dd^cu_j)^{n-1}\right)=0,
    \end{equation}
    for any $t>0$. We denote by $M_1PSH(\Omega)$ the set of negative plurisubharmonic functions in $\Omega$ satisfying $M_1$ property.\\
    
    \begin{The}\label{M1.the}
    	Let $\Omega$ be a bounded domain in $\C^n$ and $u\in PSH^{-}(\Omega)$. Then the following conditions are equivalent
    	\begin{itemize}
    		\item [(i)] $u\in M_1PSH(\Omega)$.
    		\item [(ii)] $\chi(u)\in MPSH(\Omega)$ for any  convex non-decreasing function $\chi: \R\rightarrow\R$.
    		\item [(iii)] For any open sets $U,\tilde{U}$ such that $U\Subset\tilde{U}\Subset\Omega $, for any 
    		$u_j\in PSH^{-}(\tilde{U})\cap C(\tilde{U})$ such that $u_j$ is decreasing to $u$ in $\tilde{U}$, we have
    		\begin{center}
    			$\lim\limits_{j\to\infty}\left(\int\limits_U|u_j|^{-a}(dd^cu_j)^n
    			+\int\limits_U|u_j|^{-a-1}du_j\wedge d^cu_j\wedge (dd^cu_j)^{n-1}\right)=0,$
    		\end{center}
    		for all $a>n-1$.
    	\end{itemize}
    	In particular, $M_1$ property is a local notion and $M_1PSH(\Omega)\subset MPSH(\Omega)$.
    \end{The}
    \begin{proof}
    	$(iii\Rightarrow i)$: Obvious.\\
    	$(i\Rightarrow ii)$:\\
    	Assume that $U\Subset\tilde{U}\Subset\Omega$. Let
    	$u_j\in PSH^{-}(U)\cap C(U)$ such that $u_j$ is decreasing to $u$ in $U$ and the condition \eqref{M1prop.eq} is satisfied.
    	
    	If $\chi$ is smooth and $\chi$ is constant in some interval $(-\infty, -m)$ then
    	\begin{flushleft}
    		$\begin{array}{ll}
    		(dd^c\chi(u_j))^n &=(\chi '(u_j))^n(dd^cu_j)^n+n\chi ''(u_j)(\chi '(u_j))^{n-1}du_j\wedge d^cu_j\wedge (dd^cu_j)^{n-1}\\[12pt]
    		&\leq C \boldsymbol{1}_{\{u_j>-t\}}(dd^cu_j)^n+C\boldsymbol{1}_{\{u_j>-t\}}du_j\wedge d^cu_j\wedge (dd^cu_j)^{n-1},
    		\end{array}$
    	\end{flushleft}
    	where $C,t>0$ depend only on $\chi$. Hence
    	\begin{center}
    		$\int\limits_U	(dd^c\chi(u_j))^n\stackrel{j\to\infty}{\longrightarrow}0.$
    	\end{center}
    	Then, $\chi (u)$ is maximal on $U$ for any open set $U\Subset\Omega$. Thus
    	$\chi (u)\in MPSH(\Omega)$.\\
    	In the general case, for any convex non-decreasing function $\chi$, we can find $\chi_l\searrow\chi$ such that $\chi_l$
    	is smooth, convex and $\chi|_{(-\infty, -m)}=const$ for some $m$.
    	 By above argument, $\chi_l\in MPSH(\Omega)$ for any $l\in\N$.
    	Hence $\chi (u)\in MPSH(\Omega).$\\
    	$(ii\Rightarrow iii)$:\\
    	For any $0<\alpha<\frac{1}{n}$, the function
    	\begin{center}
    		$\Phi_{\alpha}(t)=-(-t)^{\alpha}$
    	\end{center}
    	is convex and non-decreasing in $\R^-$. Assume that $u$ satisfies (ii), we have $\Phi_{\alpha}\in MPSH(\Omega)$.\\
    	
    	By \cite{Bed93} (see also \cite{Blo09}), for any $0<\alpha<\frac{1}{n}$, we have $\Phi_{\alpha}(u)\in D(\Omega)$. Then,
    	for any  $u_j\in PSH^{-}(\tilde{U})\cap C(\tilde{U})$ such that $u_j$ is decreasing to $u$ in $\tilde{U}$, we have
    	\begin{center}
    		$\int\limits_U(dd^c\Phi_{\alpha}(u_j))^n\stackrel{j\to\infty}{\longrightarrow}0, \forall 0<\alpha<\frac{1}{n},$
    	\end{center}
    	and it implies (iii).
    	
    	By $(i\Leftrightarrow ii)$, we conclude that $M_1PSH(\Omega)\subset MPSH(\Omega)$. 	Finally,  we need to show that
    	 $M_1$ property is a local notion. Assume that $u$ has local $M_1$ property. Let $U\Subset\tilde{U}\Subset\Omega$ be open sets.
    	 By the compactness of $\bar{U}$ and by the local $M_1$ property of $u$, there are open sets $U_1,...,U_m$,
    	 $\tilde{U}_1,...,\tilde{U}_m$ such that 
    	 \begin{center}
    	 	$\forall k=1,...,m: U_k\Subset\tilde{U}_k\Subset\tilde{U}$ and $u\in M_1PSH(\tilde{U}_k)$,
    	 \end{center}
     and
     \begin{center}
     	$\bar{U}\subset\bigcup\limits_{k=1}^mU_k.$
     \end{center}
 Let $u_j\in PSH^{-}(\tilde{U})\cap C(\tilde{U})$ such that $u_j$ is decreasing to $u$ in $\tilde{U}$, we have, by $(i\Leftrightarrow iii)$,
	\begin{center}
	$\lim\limits_{j\to\infty}\left(\int\limits_{U_k}|u_j|^{-a}(dd^cu_j)^n
	+\int\limits_{U_k}|u_j|^{-a-1}du_j\wedge d^cu_j\wedge (dd^cu_j)^{n-1}\right)=0,$
\end{center}
 for all $a>n-1$ and $k=1,...,m$. Hence,
	\begin{center}
	$\lim\limits_{j\to\infty}\left(\int\limits_U|u_j|^{-a}(dd^cu_j)^n
	+\int\limits_U|u_j|^{-a-1}du_j\wedge d^cu_j\wedge (dd^cu_j)^{n-1}\right)=0, \forall a>n-1.$
\end{center}
Then $u$ satisfies $(iii)$. Thus $u$ has $M_1$ property.
    \end{proof}
    The following proposition is an immediately corollary of Theorem \ref{M1.the}
    \begin{Prop}
    	Let $\Omega$ be a bounded domain in $\C^n$.
    	\begin{itemize}
    		\item [(i)] If $u\in M_1PSH(\Omega)$ then $\chi (u)\in M_1PSH(\Omega)$ 
    		for any  convex non-decreasing function $\chi: \R^-\rightarrow\R^-$.
    		\item [(ii)] If $u_j\in M_1PSH(\Omega)$ and $u_j$ is decreasing to $u$ then $u\in M_1PSH(\Omega)$.
    		\item[(iii)] Let $u\in PSH^-(\Omega)\cap C^2(\Omega\setminus F)$, where $F=\{z: u(z)=-\infty\}$ is closed. If
    		\begin{center}
    			$(dd^c u)^n=du\wedge d^cu\wedge (dd^c u)^{n-1}=0$
    		\end{center}
    		in $\Omega\setminus F$ then $u\in M_1PSH(\Omega)$.
    	\end{itemize}
    \end{Prop}
    In some special cases, we can easily check $M_1$ property by the following criteria
    \begin{Prop}\label{M1crit.prop}
    	Let $\Omega$ be a bounded domain in $\C^n$. Let $\chi:\R\rightarrow\R$ be a smooth convex increasing function such that
    	$\chi'' (t)>0$ for any $t\in\R$. Assume also that $\chi$ is lower bounded.
    	 If $u\in PSH^-(\Omega)$ and $\chi (u)\in MPSH(\Omega)$ then $u\in M_1PSH(\Omega)$.
    \end{Prop}
    \begin{proof}
    	Let $U\Subset\tilde{U}\Subset\Omega$ and $u_j\in PSH(\tilde{U})\cap C(\tilde{U})$ such that $u_j$ is decreasing to $u$. Then
    	\begin{center}
    		$dd^c(\chi(u_j))=\chi'(u_j)dd^cu_j+\chi''(u_j)du_j\wedge d^cu_j$
    	\end{center}
    	and
    	\begin{center}
    		$	(dd^c\chi(u_j))^n=(\chi '(u_j))^n(dd^cu_j)^n+n\chi''(u_j)(\chi'(u_j))^{n-1} du_j\wedge d^cu_j\wedge (dd^c u_j)^{n-1}.$
    	\end{center}
    For any $t>0$, there exists $C>0$ depending only on $t$ and $\chi$ such that
    \begin{equation}\label{M1crit.eq1}
    	(dd^c\chi(u_j))^n
    	\geq C \boldsymbol{1}_{\{u_j>-t\}}(dd^cu_j)^n+C \boldsymbol{1}_{\{u_j>-t\}} du_j\wedge d^cu_j\wedge (dd^c u_j)^{n-1}.
    \end{equation}
    Note that $\chi (u)\in D(\Omega)\cap MPSH(\Omega)$. Hence
    \begin{equation}\label{M1crit.eq2}
    \lim\limits_{j\to\infty}\int\limits_U(dd^c\chi(u_j))^n=0.
    \end{equation}	
    Combining \eqref{M1crit.eq1} and \eqref{M1crit.eq2}, we have
    \begin{center}
    	$\lim\limits_{j\to\infty}\left(\int\limits_{U\cap\{u_j>-t\}}(dd^c u_j)^n
    	+\int\limits_{U\cap\{u_j>-t\}}du_j\wedge d^cu_j\wedge(dd^cu_j)^{n-1}\right)=0.$
    \end{center}
    Thus $u\in M_1PSH(\Omega)$.
    \end{proof}
    \begin{Ex}
    	(i) If $u$ is a negative plurisubharmonic function in $\Omega\subset\C^n$ depending only on $n-1$ variables then $u$ has $M_1$
    	property.\\
    	(ii) If $f: \Omega\rightarrow \C^n$ is a holomorphic mapping of rank $<n$ then $(dd^c|f|^2)^n=0$
    	 (see, for example, in \cite{Ras98}). Then, by Proposition \ref{M1crit.prop}, $\log |f|\in M_1PSH(\Omega)$
    	 if it is negative in $\Omega$.
    \end{Ex}
   \subsection{Doubling variables technique} 
   \begin{The}\label{app1}
   	If $u,v\in MPSH(\Omega)$ and $\chi: \R\rightarrow\R$ is a convex non-decreasing function then
   	$(z,w)\mapsto\chi (u(z)+v(w))\in MPSH(\Omega\times\Omega)$.
   \end{The}
\begin{proof}
   	Without loss of generality, we can assume that $u,v\in PSH^-(\Omega)$. 
   	
   	Let $z_0, w_0\in\Omega$ and open balls $U, \tilde{U}, V, \tilde{V}$ such that
   	$z_0\in U\Subset\tilde{U}\Subset\Omega$, $w_0\in V\Subset\tilde{V}\Subset\Omega$. We have $u\in MPSH(\tilde{U})$
   	and $v\in MPSH(\tilde{V})$. We will show that $u(z)+v(w)$ have $M_1$ property in $U\times V$.
   	
   	Let $u_j\in PSH^-(\tilde{U})\cap C(\tilde{U})$ and $v_j\in PSH^-(\tilde{V})\cap C(\tilde{V})$ such that
   	$u_j$ is decreasing to $u$ in $\tilde{U}$ and $v_j$ is decreasing to $v$ in $\tilde{V}$. By \cite{Wal68},
   	there are $\tilde{u}_j\in PSH^-(\tilde{U})\cap C(\tilde{U})$, $\tilde{v}_j\in PSH^-(\tilde{V})\cap C(\tilde{V})$
   	such that
   	\begin{center}
   		$\begin{cases}
   		\tilde{u_j}=u_j\quad\mbox{ in }\tilde{U}\setminus U,\\
   		\tilde{v_j}=v_j\quad\mbox{ in }\tilde{V}\setminus V,\\
   		(dd^c\tilde{u}_j)^n=0\quad\mbox{ in }U,\\
   		(dd^c\tilde{v}_j)^n=0\quad\mbox{ in }V.
   		\end{cases}$
   	\end{center}
   	By the maximality of $u$ and $v$, we conclude that $\tilde{u}_j$ is decreasing to $u$ in $\tilde{U}$ and
   	$\tilde{v}_j$ is decreasing to $v$ in $\tilde{V}$. In $U\times V$, we have
   	\begin{flushleft}
   		$(dd^c(\tilde{u}_j(z)+\tilde{v}_j(w)))^{2n}=C^n_{2n}(dd^c\tilde{u}_j)^n_z\wedge (dd^c \tilde{v}_j)^n_w=0$\\[12pt]
   		$d(\tilde{u}_j(z)+\tilde{v}_j(w))\wedge d^c(\tilde{u}_j(z)+\tilde{v}_j(w))\wedge (dd^c(\tilde{u}_j(z)+\tilde{v}_j(w)))^{2n-1}$\\[12pt]
   		$= C^{n-1}_{2n-1}d_z\tilde{u}_j\wedge d^c_z\tilde{u}_j\wedge (dd^c \tilde{u}_j)^{n-1}_z\wedge (dd^c  \tilde{v}_j)^n_w$
   		$+C^{n-1}_{2n-1}d_w \tilde{v}_j\wedge d^c_w \tilde{v}_j\wedge (dd^c  \tilde{v}_j)^{n-1}_w\wedge (dd^c \tilde{u}_j)^n_z$\\[12pt]
   		$=0.$
   	\end{flushleft}
   	
   	Then $u(z)+v(w)$ has $M_1$ property in $U\times V$. By Theorem \ref{M1.the}, $M_1$ property is a local notion. Hence
   	$u(z)+v(w)\in M_1PSH(\Omega\times\Omega)$. And it implies that $u(z)+v(w)\in MPSH_{\chi}(\Omega\times\Omega)$
   	for any convex non-decreasing function $\chi$.
   	\end{proof}
   \subsection{Proof of Theorem \ref{main1}}
   Without loss of generality, we can assume that $u, u_j<0$. We need to show that 
   \begin{equation}\label{main1modified.eq}
   \int\limits_U|u_j|^{-a}(dd^cu_j)^n\stackrel{j\to\infty}{\longrightarrow}0,\,\forall a>n-1.
   \end{equation}
   	Let $v=|z_1|^2+...+|z_{n-1}|^2+x_n+y_n-M$, where $M=\sup\limits_{\Omega}(|z|^2+|x_n|+|y_n|)$. Then $v\in MPSH(\Omega)$.
   	By Theorem \ref{app1}, $\chi(u(z)+v(w))\in MPSH(\Omega\times\Omega)$ for any convex non-decreasing function $\chi$.
   	
   	By \cite{Bed93},\cite{Blo09}, for any $0<\alpha<\frac{1}{2n}$, we have $\Phi_{\alpha}(u(z)+v(w))\in D(\Omega\times\Omega)$,
   	where $\Phi_{\alpha}$ is defined as in the proof of Theorem \ref{M1.the}.\\
   	Then
   	\begin{center}
   		$\int\limits_{U\times U}(dd^c\Phi(u_j(z)+v(w)))^{2n}\stackrel{j\to\infty}{\longrightarrow}0,$
   	\end{center}
   	for any $0<\alpha<\frac{1}{2n}$. Hence
   	\begin{equation}\label{1proofapp2.eq}
   	\int\limits_U|u_j|^{-2n-1+2n\alpha}(dd^cu_j)^n\stackrel{j\to\infty}{\longrightarrow}0,\,\forall 0<\alpha<\frac{1}{2n}.
   	\end{equation}
   	Moreover, $\Phi_{\beta}(u)\in D(\Omega)$ for any $0<\beta<\frac{1}{n}$. Then, for any $0<\beta<\frac{1}{n}$,
   	there is $C_{\beta}>0$ such that
   	\begin{center}
   		$\int\limits_U(dd^c\Phi_{\beta}(u_j))^n\leq C_{\beta},\,\forall j>0.$
   	\end{center}
   	Hence
   	\begin{equation}\label{2proofapp2.eq}
   	\int\limits_U|u_j|^{-n+n\beta}(dd^cu_j)^n\leq C_{\beta},\,\forall j>0, \forall 0<\beta<\frac{1}{n}.
   	\end{equation}
   	Combining \eqref{1proofapp2.eq}, \eqref{2proofapp2.eq} and using H\"older inequality, we obtain \eqref{main1modified.eq}.
   \section{A sufficient condition for maximality}
   The purpose of this section is to give a proof of Theorem \ref{main2}. In order to prove Theorem \ref{main2}, we need to use
   the comparison principle.
   \subsection{Comparison principle} In \cite{BT82}, Bedford and Taylor have proved the following comparison principle
   \begin{The}\label{the.compa}
   			Let $\Omega\subset\C^n$ be a bounded domain. Let $u, v\in PSH(\Omega)\cap L^{\infty}(\Omega)$ such that
   			\begin{center}
   				$\liminf\limits_{\Omega\ni z\to\partial \Omega}(u(z)-v(z))\geq 0$.
   			\end{center}
   			Then
   			\begin{center}
   				$\int\limits_{u<v}(dd^c v)^n\leq\int\limits_{u<v}(dd^c u)^n.$
   			\end{center}
   	\end{The}
   	This theorem has been generalized by Xing (1996) and then by Nguyen Van Khue and Pham Hoang Hiep (2009):
   	\begin{The}\label{the.xingcompa}\cite{Xing96}
   		Let $\Omega\subset\C^n$ be a bounded domain. Let $u, v$ be bounded plurisubharmonic functions in $\Omega$ such that
   		 \begin{center}
   		 	$\liminf\limits_{\Omega\ni z\to\partial \Omega}(u(z)-v(z))\geq 0$.
   		 \end{center}
   		 Then, for any constant $r\geq 1$ and all $w_1, ..., w_n\in PSH(\Omega, [0, 1])$,
   		 \begin{center}
   		 	$\dfrac{1}{(n!)^2}\int\limits_{\{u<v\}}(v-u)^ndd^cw_1\wedge...\wedge dd^cw_n
   		 	+\int\limits_{\{u<v\}}(r-w_1)(dd^cv)^n
   		 	\leq \int\limits_{\{u<v\}}(r-w_1)(dd^cu)^n.$
   		 \end{center}
   	\end{The}
   	\begin{The}\label{the.npcompa}\cite{NP09}
   			Let $\Omega\subset\C^n$ be a bounded domain. Let $u, v\in \mathcal{E}(\Omega)$. Assume that
   			\begin{center}
   				$\liminf\limits_{\Omega\ni z\to\partial \Omega}(u(z)-v(z))\geq 0$.
   			\end{center}
   			Then, for any $r\geq 1$,
   			\begin{center}
   					$\dfrac{1}{k!}\int\limits_{\{u<v\}}(v-u)^kdd^cw_1\wedge...\wedge dd^cw_n
   					+\int\limits_{\{u<v\}}(r-w_1)(dd^cv)^k\wedge dd^cw_{k+1}\wedge...\wedge dd^cw_n
   					\leq\int\limits_{\{u<v\}\cup\{u=v=-\infty\}}(r-w_1)(dd^cu)^k\wedge dd^cw_{k+1}\wedge...\wedge dd^cw_n,$
   			\end{center}
   			where $1\leq k\leq n$, $w_1,...,w_k\in PSH(\Omega, [0,1])$, $w_{k+1},...,w_n\in \mathcal{E}(\Omega)$.
   	\end{The}
   	In this section, we will use another improved version of Theorem \ref{the.compa}:
   	\begin{The}\label{the.newcompa}
   		Let $\Omega\subset\C^n$ be a bounded domain. Let $u, v\in PSH(\Omega)\cap L^{\infty}(\Omega)$ such that
   		\begin{center}
   			$\liminf\limits_{\Omega\ni z\to\partial \Omega}(u(z)-v(z))\geq 0$.
   		\end{center}
   		Assume that $T$ is a closed, positive current of bidegree $(n-q, n-q)$ in $\Omega$. Then
   		\begin{center}
   			$\int\limits_{u<v}(dd^c v)^q\wedge T\leq\int\limits_{u<v}(dd^c u)^q\wedge T.$
   		\end{center}
   	\end{The}
   	The proof of Theorem \ref{the.newcompa} is straight forward from the classical result. The details of arguments can be found in
   	\cite{BT82} and some books (for example, \cite{Kli91} and \cite{Kol05}). For the reader's convenience, we also give the details
   	in the appendix.
   \subsection{Proof of Theorem \ref{main2}}
   	Without loss of generality, we can assume that $\Omega$ is bounded and $u_{jk}\geq u+\dfrac{1}{k}$ for any $j, k$.
   We will show that, for any open set $G\Subset\Omega$, if $v\in PSH(\Omega)$ satisfies
   $v\leq u$ on $\Omega\setminus G$ then $\int\limits_{\{u<v\}}T_j\wedge\omega^q=0$ for any $j\in J$.\\
   Denote $M=\sup\limits_{\Omega}|z|^2$ . For any $\epsilon>0$ and $j\in J$, we define
   \begin{center}
   	$v_{jk}(z)=
   	\begin{cases}
   	u_{jk}(z)\qquad\mbox{if}\quad z\in\Omega\setminus G,\\
   	\max {\{u_{jk} (z), v(z)+\epsilon (|z|^2-M)\}}\qquad\mbox{if}\quad z\in G.
   	\end{cases}$
   \end{center}
   By the comparison principle, we have,
   \begin{center}
   	$\int\limits_{\{u_{jk}<v_{jk}\}}(dd^c v_{jk})^q\wedge T_j\leq \int\limits_{\{u_{jk}<v_{jk}\}}(dd^c u_{jk})^q\wedge T_j.$
   \end{center}
   Hence
   \begin{center}
   	$\epsilon^q\int\limits_{\{u_{jk}<v +\epsilon(|z|^2-M)\}}\omega^q\wedge T_j\leq \int\limits_{\Omega}(dd^c u_{jk})^q\wedge T_j.$
   \end{center}
   Letting $k\rightarrow\infty$, we get 
   \begin{center}
   $\int\limits_{\{u<v +\epsilon(|z|^2-M)\}}\omega^q\wedge T_j=0$.
\end{center}
   Since $\epsilon$ and $j$ are arbitrary,
   we obtain 
   \begin{center}
   	$\int\limits_{\{u<v\}}\omega^q\wedge T_j=0$,
   \end{center} 
   for any $j\in J$. By the assumption,  we have $\lambda (\{u<v\})=0$.
   Hence $v\leq u$ in $\Omega$.\\
   The proof is completed.
   \section{On the Blocki's example}
  In this section, we give a proof of Theorem \ref{main3}. In calculating, we use the notation $d^c=i (\bar{\partial}-\partial)$.\\
  
  Denote $f(z)=-\sqrt{-\log |z|}$. We have, for any $z\in\Delta\setminus\{0\}$,
  \begin{center}
  	$\dfrac{\partial f}{\partial z}(z)=\dfrac{1}{4z\sqrt{-\log |z|}},$
  \end{center}
  and
  \begin{center}
  	$\dfrac{\partial^2 f}{\partial z\partial \bar{z}}(z)=\dfrac{1}{16 |z|^2(\sqrt{-\log |z|})^3}.$
  \end{center}
  Then
  \begin{equation}\label{eq.ff'ff''.sec3}
  f.\dfrac{\partial^2 f}{\partial z\partial \bar{z}}=-\left( \dfrac{\partial f}{\partial z}\right)^2.
  \end{equation}
  Note that $u(z, w)=-f(z).f(w)$. Hence, for any $(z, w)\in (\Delta\setminus\{0\})^2$,
  \begin{center}
  	$\partial u=-\dfrac{\partial f}{\partial z}(z). f(w)dz-\dfrac{\partial f}{\partial w}(w). f(z)dw,$
  \end{center}
  and
  \begin{center}
  	$	du\wedge d^c u  =2i \partial u\wedge \bar{\partial} u
  	=|\dfrac{\partial f}{\partial z}(z)|^2.(f(w))^2.2idz\wedge d\bar{z}
  	+|\dfrac{\partial f}{\partial w}(w)|^2.(f(z))^2.2idw\wedge d\bar{w}
  	+\dfrac{\partial f}{\partial z}(z).\dfrac{\partial f}{\partial \bar{w}}(w)f(z)f(w).2i dz\wedge d\bar{w}
  	+\dfrac{\partial f}{\partial w}(w).\dfrac{\partial f}{\partial \bar{z}}(z)f(z)f(w).2i dw\wedge d\bar{z},$\\
  \end{center}
  and
  \begin{center}
  	$	dd^c u =2i\partial\bar{\partial}u
  =-\dfrac{\partial^2 f}{\partial z\partial\bar{z}}(z)f(w).2i dz\wedge d\bar{z}
  -\dfrac{\partial^2 f}{\partial w\partial\bar{w}}(w)f(z).2i dw\wedge d\bar{w}
  -\dfrac{\partial f}{\partial z}(z)\dfrac{\partial f}{\partial \bar{w}}(w).2idz\wedge d\bar{w}
  -\dfrac{\partial f}{\partial w}(w)\dfrac{\partial f}{\partial \bar{z}}(z).2idw\wedge d\bar{z}.$\\
  \end{center}
  Then
  \begin{center}
  	$d u\wedge d^c u\wedge dd^c u=(2|\dfrac{\partial f}{\partial z}(z)|^2.|\dfrac{\partial f}{\partial w}(w)|^2f(z)f(w)
  	-|\dfrac{\partial f}{\partial z}(z)|^2.(f(w))^2.f(z)\dfrac{\partial^2 f}{\partial w\partial\bar{w}}(w)
  	-|\dfrac{\partial f}{\partial z}(w)|^2.(f(z))^2.f(w)\dfrac{\partial^2 f}{\partial z\partial\bar{z}}(z))
  	(2i dz\wedge d\bar{z})\wedge (2i dw\wedge d\bar{w}).$
  \end{center}
  By \eqref{eq.ff'ff''.sec3}, we have
  \begin{center}
  	$d u\wedge d^c u\wedge dd^c u=64 |\dfrac{\partial f}{\partial z}(z)|^2.|\dfrac{\partial f}{\partial w}(w)|^2f(z)f(w)dV_4
  	=\dfrac{dV_4}{4\sqrt{\log |z|\log |w|}|z|^2|w|^2},$
  \end{center}
  where $dV_4$ is the standard volume form in $\R^4$.\\
  
  Let $\chi : (-\infty, 0)\rightarrow (-\infty, 0)$ be a smooth, convex, non-decreasing function such that $\chi$ is constant
  in some interval $(-\infty, -M)$. We have
  \begin{center}
  	$d\chi (u)=\chi '(u)du, \qquad d^c\chi (u)=\chi'(u)d^c u,$
  \end{center}
  and
  \begin{center}
  	$dd^c\chi (u)=\chi ''(u)du\wedge d^cu+\chi '(u)dd^c u$.
  \end{center}
  Then
  \begin{flushleft}
  	$\begin{array}{ll}
  	(dd^c\chi (u))^2 &=(\chi'(u))^2(dd^c u)^2+2\chi'(u)\chi''(u)du\wedge d^cu\wedge dd^cu\\
  	&=2\chi'(u)\chi''(u)du\wedge d^cu\wedge dd^cu\\
  	&=\dfrac{\chi'(u)\chi''(u)}{2\sqrt{\log |z|\log |w|}|z|^2|w|^2}dV_4.
  	\end{array}$
  \end{flushleft}
Let $w_0\in\Delta\setminus\{0\}$ and $0<2r<\min\{|w_0|, 1-|w_0|\}$. We will estimate
\begin{center}
	$I:=I(\chi, w_0, r):=\int\limits_{\Delta_{r}\times\Delta_{r}(w_0)}\dfrac{\chi'(u)\chi''(u)}{2\sqrt{\log |z|\log |w|}|z|^2|w|^2}dV_4,$
\end{center}
where $\Delta_{r}=\{z\in\C: |z|<r\}$, $\Delta_{r} (w_0)=\{w\in\C: |w-w_0|<r\}$.\\
By using Fubini's theorem and changing to polar coordinates, we get
\begin{center}
	$I=\int\limits_{\Delta_{r}(w_0)}
	\int\limits_0^r\dfrac{2\pi\chi'(-\sqrt{-\log |w|}\sqrt{-\log t})\chi''(-\sqrt{-\log |w|}\sqrt{-\log t})}
	{\sqrt{-\log |w|}|w|^2.\sqrt{-\log t}.t}dt dV_2(w).$
\end{center}
Using the substitution $s=-\sqrt{-\log t}$, we have
\begin{flushleft}
	$\begin{array}{ll}
	I(\chi, w_0, r) &=\int\limits_{\Delta_{r}(w_0)}\dfrac{4\pi}{\sqrt{-\log |w|}|w|^2}
	\int\limits_{-\infty}^{-\sqrt{-\log r}}\chi'(\sqrt{-\log |w|}s)\chi''(\sqrt{-\log |w|}s)ds dV_2(w)\\
	&=\int\limits_{\Delta_{r}(w_0)}\dfrac{2\pi}{(-\log |w|)|w|^2}
\int\limits_{-\infty}^{-\sqrt{-\log r}}\dfrac{d(\chi'(\sqrt{-\log |w|}s))^2}{ds}dsdV_2(w)\\
&=\int\limits_{\Delta_{r}(w_0)}\dfrac{2\pi(\chi'(-\sqrt{\log |w|\log r}))^2}{(-\log |w|)|w|^2}dV_2(w)\\
&\geq \int\limits_{\Delta_{r}(w_0)}\dfrac{2\pi(\chi'(-\sqrt{\log (|w_0|-r)\log r}))^2}{(-\log (|w_0|-r))(|w_0|+r)^2}dV_2(w)\\
&\geq \dfrac{4\pi^2(\chi'(\log r))^2}{9(-\log r).r^2}.\\
	\end{array}$
\end{flushleft}
Now, we assume by contradiction that there exists a sequence of bounded convex non-decreasing functions
 $\chi_m : (-\infty, 0)\rightarrow (-\infty, 0)$  such that $\chi_m\searrow Id$ and $(dd^c\chi_m(u))^2\stackrel{w}{\rightarrow}0$
 as $m\rightarrow\infty$. Observe that for any $\chi_m$, there exists a sequence of smooth, convex, non-decreasing function
 $\chi_{mk}$ such that $\chi_{mk}\searrow\chi_m$ and $\chi_{mk}=const$ in some interval $(-\infty, -M_{mk})$. Replacing $\chi_m$ by 
 $\chi_{mk_m}$ with $k_m\gg 1$, we can assume that $\chi_m$ is smooth and $\chi_m=const$ in some interval $(-\infty, -M_m)$.
 
 Since $\chi_m\searrow Id$, we have $\lim\limits_{m\to\infty}\chi'(t)=1$ for all $t\in (-\infty, 0)$.
  Then, for any $w_0\in\Delta\setminus\{0\}$ and $0<2r<\min\{|w_0|, 1-|w_0|\}$,
  \begin{center}
  	$0=\lim\limits_{m\to\infty}(dd^c\chi_m(u))^2\geq\lim\limits_{m\to\infty}\dfrac{4\pi^2(\chi_m'(\log r))^2}{9(-\log r).r^2}
  	=\dfrac{4\pi^2}{9(-\log r).r^2}>0.$
  \end{center}
 We have a contradiction. The proof is completed.
 \appendix
 \section{Comparison principle}
 In this appendix, we will give the details of a proof of the following theorem:
 \begin{The}\label{the.newcompa.appendix}
 	Let $\Omega\subset\C^n$ be a bounded domain. Let $u, v\in PSH(\Omega)\cap L^{\infty}(\Omega)$ such that
 	\begin{center}
 		$\liminf\limits_{\Omega\ni z\to\partial \Omega}(u(z)-v(z))\geq 0$.
 	\end{center}
 	Assume that $T$ is a closed, positive current of bidegree $(n-q, n-q)$ in $\Omega$. Then
 	\begin{center}
 		$\int\limits_{\{u<v\}}(dd^c v)^q\wedge T\leq\int\limits_{\{u<v\}}(dd^c u)^q\wedge T.$
 	\end{center}
 \end{The}
 The arguments used in this appendix are based on \cite{BT82}, \cite{Ceg88}, \cite{Kli91}, \cite{Dem}, \cite{Kol05}.
 \begin{Lem}\label{lem.weakconvergence}
 	Let $u_0,...,u_q$ be locally bounded plurisubharmonic functions in a domain $\Omega\subset\C^n$ ($0\leq q\leq n$) 
 	and let $u_0^k,...,u_q^k$ be decreasing sequences of plurisubharmonic functions converging pointwise to $u_0,...,u_q$. Let $T$ be
 	a closed, positive current of bidegree $(n-q, n-q)$ in $\Omega$. Then
 	$u_0^kdd^cu_1^k\wedge ...\wedge dd^cu_q^k\wedge T$ converges weakly to $u_0dd^cu_1\wedge ...\wedge dd^cu_q\wedge T.$
 \end{Lem}
 \begin{Lem}\label{lem.gradient} Let $\Omega\subset\C^n$ be a domain. If $u, v\in C^2(\Omega)$ and
 	$T$ is a  positive current of bidegree $(n-1, n-1)$ in $\Omega$ satisfying
 	\begin{center}
 		$\int\limits_{\Omega}\omega\wedge T<\infty,$
 	\end{center}
 	then 
 	\begin{center}
 		$|\int\limits_{\Omega}du\wedge d^cv\wedge T|^2\leq \int\limits_{\Omega}du\wedge d^cu\wedge T
 		\int\limits_{\Omega}dv\wedge d^cv\wedge T,$
 	\end{center}
 	where $\omega=dd^c |z|^2$.
 \end{Lem}
 \begin{The}\label{the.capconver}
 	Let $T$ be a closed, positive current of bidegree $(n-q, n-q)$ in $B_2:=B(0, 2)\subset\C^n$ ($1\leq q\leq n$). 
 	If $u_j\in PSH(B_2)\cap L^{\infty}(B_2)$
 	is a decreasing sequence, convergent to $u\in PSH(B_2)\cap L^{\infty}(B_2)$, then
 	\begin{center}
 		$\lim\limits_{j\to\infty}\left(\sup \{\int\limits_{\overline{B}_1}(u_j-u)(dd^c v)^q\wedge T: v\in PSH(B_2), 0\leq v\leq 1
 		\}\right)=0.$
 	\end{center}
 \end{The}
 \begin{proof}
 	Without loss of generality, we can assume that $0<u_j, u<1$, $u_j$ are smooth and $u_j=u$ in $B_2\setminus\bar{B_1}$.\\
 	Take $v\in PSH(\Omega, (0, 1))$. We define $\varrho (z)=10 |z|^2-20$ and
 	\begin{center}
 		$w=\begin{cases}
 		\max\{\varrho, \frac{v+1}{2}\}\quad\mbox{in}\quad B_2,\\
 		\varrho\quad\mbox{in}\quad \C^n\setminus B_2,
 		\end{cases}$\\
 		$\tilde{w}=\max\{\varrho, 1\}.$
 	\end{center}
 	Then, $w, \tilde{w}\in PSH(\C^n)$ and $w=\tilde{w}=\varrho$ in $\C^n\setminus B_{3/2}$. If $\chi_{\epsilon}$ is the standard
 	smooth kernel and $0<\epsilon\ll 1$ then $w_{\epsilon}:=w*\chi_{\epsilon}=\tilde{w}*\chi_{\epsilon}=:\tilde{w}_{\epsilon}$
 	in $\C^n\setminus B_{3/2}$ and $\tilde{w}_{\epsilon}=1$ in a neighborhood of $\bar{B}_1$. \\
 	We have, for and $k,j\in\Z^+$,
 	\begin{flushleft}
 		$\begin{array}{ll}
 		\int\limits_{\bar{B}_1}(u_j-u_k)(dd^cw_{\epsilon})^q\wedge T
 		& =	\int\limits_{B_2}(u_j-u_k)(dd^cw_{\epsilon})^q\wedge T\\
 		& =-\int\limits_{B_2}d(u_j-u_k)d^cw_{\epsilon}(dd^cw_{\epsilon})^{q-1}\wedge T\\
 		& =-\int\limits_{B_2}d(u_j-u_k)d^c(w_{\epsilon}-\tilde{w}_{\epsilon})(dd^cw_{\epsilon})^{q-1}\wedge T\\
 		& =-\int\limits_{\bar{B}_1}d(u_j-u_k)d^c(w_{\epsilon}-\tilde{w}_{\epsilon})(dd^cw_{\epsilon})^{q-1}\wedge T.\\
 		\end{array}$
 	\end{flushleft}
 	By using Lemma \ref{lem.gradient} and Chern-Levine-Nirenberg theorem, we have, for any $k\geq j$,
 	\begin{flushleft}
 		$|\int\limits_{\bar{B}_1}d(u_j-u_k)d^c(w_{\epsilon}-\tilde{w}_{\epsilon})(dd^cw_{\epsilon})^{q-1}\wedge T|$\\
 		$\leq \left(\int\limits_{\bar{B}_1}d(u_j-u_k)d^c(u_j-u_k)(dd^cw_{\epsilon})^{q-1}\wedge T\right)^{1/2}$\\
 		$\times \left(\int\limits_{\bar{B}_1}d(w_{\epsilon}-\tilde{w}_{\epsilon})d^c(w_{\epsilon}-\tilde{w}_{\epsilon})
 		(dd^cw_{\epsilon})^{q-1}\wedge T\right)^{1/2}$\\
 		$\leq \left(-\int\limits_{\bar{B}_1}(u_j-u_k)dd^c(u_j-u_k)(dd^cw_{\epsilon})^{q-1}\wedge T\right)^{1/2}$\\
 		$\times \left(-\int\limits_{B_{3/2}}(w_{\epsilon}-\tilde{w}_{\epsilon})dd^c(w_{\epsilon}-\tilde{w}_{\epsilon})
 		(dd^cw_{\epsilon})^{q-1}\wedge T\right)^{1/2}$\\
 		$\leq \left(\int\limits_{\bar{B}_1}(u_j-u_k)dd^cu_k(dd^cw_{\epsilon})^{q-1}\wedge T\right)^{1/2}$\\
 		$\times \left(\int\limits_{B_{3/2}}
 		(dd^cw_{\epsilon})^q\wedge T\right)^{1/2}$\\
 		$\leq C. \|T\|_{B_{5/3}}^{1/2}\left(\int\limits_{\bar{B}_1}(u_j-u_k)dd^cu_k\wedge(dd^cw_{\epsilon})^{q-1}\wedge T\right)^{1/2},$\\
 	\end{flushleft}
 	where $C>0$ depends only on $n, q$. Repeating this argument $q-1$ times, we get
 	\begin{center}
 		$\int\limits_{\bar{B}_1}(u_j-u_k)(dd^cw_{\epsilon})^q\wedge T
 		\leq A\left(\int\limits_{\bar{B}_1}(u_j-u_k)(dd^cu_k)^q\wedge T\right)^B.$
 	\end{center}
 	where $A, B>0$ depend only on $n, q$. Hence
 	\begin{center}
 		$\int\limits_{\bar{B}_1}(u_j-u_k)(dd^c v)^q\wedge T
 		=2^n\lim\limits_{\epsilon\searrow 0}\int\limits_{\bar{B}_1}(u_j-u_k)(dd^cw_{\epsilon})^q\wedge T
 		\leq 2^nA\left(\int\limits_{\bar{B}_1}(u_j-u_k)(dd^cu_k)^q\wedge T\right)^B.$
 	\end{center}
 	Letting $k\rightarrow\infty$ and taking the supremum with respect to $v\in PSH(B_2, (0, 1))$, we get
 	\begin{center}
 		$\sup\limits_v\int\limits_{\bar{B}_1}(u_j-u)(dd^cv)^q\wedge T
 		\leq 2^nA\left(\int\limits_{\bar{B}_1}(u_j-u)(dd^cu)^q\wedge T\right)^B.$
 	\end{center}
 	By using  Lebesgue's dominated convergence theorem, we obtain
 	\begin{center}
 		$\lim\limits_{j\to\infty}\left(\sup \{\int\limits_{\overline{B}_1}(u_j-u)(dd^c v)^q\wedge T: v\in PSH(B_2), 0\leq v\leq 1
 		\}\right)=0.$
 	\end{center}
 \end{proof}
 If $\Omega\subset\C^n$ is a domain, $T$ is a closed, positive current of bidegree $(n-q, n-q)$ in $\Omega$ and $K$ is a compact subset
 of $\Omega$, we define
 \begin{center}
 	$C(K, \Omega, T)= \sup \{\int\limits_{K}(dd^c v)^q\wedge T: v\in PSH(\Omega), 0\leq v\leq 1\}.$
 \end{center}
 If $U\subset\Omega$, we put 
 \begin{center}
 	$C(U, \Omega, T)=\sup \{C(K, \Omega, T): K$ is a compact subset of $ U\}.$
 \end{center}
 \begin{Cor}\label{cor.quasicontinuous}
 	Let $u\in PSH(\Omega)\cap L^{\infty}(\Omega)$, where $\Omega\subset \C^n$ is a domain. 
 	Let $T$ be a closed, positive current of bidegree $(n-q, n-q)$ in $\Omega$, where $1\leq q\leq n$. For each $\epsilon>0$, 
 	there exists an open set $U\subset\Omega$ such that $C( U, \Omega, T)<\epsilon$ and the restriction of $u$
 	in $\Omega\setminus U$ is continuous.
 \end{Cor}
 \begin{Lem}\label{lem.compacontinuous}
 	The theorem \ref{the.newcompa.appendix} is true when $u$ and $v$ are continuous.
 \end{Lem}
 \begin{proof}
 	In this case we may suppose that for each $w\in\partial\Omega$,
 	\begin{center}
 		$\lim\limits_{\Omega\ni z\to w}(u(z)-v(z))=0,$
 	\end{center}
 	and $u<v$ in $\Omega$. For $\epsilon>0$, we define $v_{\epsilon}=\max\{v-\epsilon, u\}$. Then $v_{\epsilon}=u$ near $\partial\Omega$.\\
 	
 	Given $\epsilon>0$, choose a non-negative function $\phi\in C_0^{\infty}(\Omega)$ such that $\phi\equiv 1$ 
 	in a neighbourhood of the closure of the set where $v_{\epsilon}>u$. Then
 	\begin{flushleft}
 		$\begin{array}{ll}
 		\int\limits_{\Omega}\phi (dd^c v_{\epsilon})^q\wedge T
 		&=\int\limits_{\Omega}v_{\epsilon}dd^c\phi\wedge (dd^c v_{\epsilon})^{q-1}\wedge T\\
 			&=\int\limits_{\Omega}udd^c\phi\wedge (dd^c u)^{q-1}\wedge T\\
 			&=\int\limits_{\Omega}\phi (dd^c u)^q\wedge T.
 		\end{array}$
 	\end{flushleft}
 	Hence
 	\begin{center}
 		$\int\limits_{\Omega} (dd^c v_{\epsilon})^q\wedge T=\int\limits_{\Omega} (dd^c u)^q\wedge T.$
 	\end{center}
 	Moreover, for any $\psi\in C_0(\Omega, [0,1])$,
 	\begin{center}
 		$\int\limits_{\Omega}\psi (dd^c v)^q\wedge T=
 		\lim\limits_{\epsilon\to 0}\int\limits_{\Omega}\psi (dd^c v_{\epsilon})^q\wedge T
 		\leq \int\limits_{\Omega} (dd^c v_{\epsilon})^q\wedge T.$
 	\end{center}
 	Thus,
 	\begin{center}
 		$\int\limits_{\Omega}(dd^c v)^q\wedge T\leq\int\limits_{\Omega}(dd^c u)^q\wedge T.$
 	\end{center}
 \end{proof}
 \begin{proof}[Proof of Theorem \ref{the.newcompa.appendix}]
 	By considering $u+2\delta$ instead of $u$ and then letting $\delta\searrow 0$, we can assume that
 	\begin{center}
 		$\liminf\limits_{\Omega\ni z\to\partial \Omega}(u(z)-v(z))\geq 2\delta>0$,
 	\end{center}
 	and it remains to show that
 	\begin{center}
 		$\int\limits_{\{u<v\}}(dd^c v)^q\wedge T\leq\int\limits_{\{u<v+\delta\}}(dd^c u)^q\wedge T.$
 	\end{center}
 	Then the set $S=\{u<v+\delta\}$ is relatively compact in $\Omega$. 
 	Therefore we can find decreasing sequences of continuous plurisubharmonic functions $u_j, v_j$ in an open neighborhood $\Omega_1$ of
 	$\bar{S}$ such that $\lim u_j=u, \lim v_j=v$ and $u_j\geq v_j$ on $\Omega_1\setminus\overline{\Omega_2}$ for all $j$,
 	where $\Omega_2$ is open such that $\bar{S}\subset\Omega_2\Subset\Omega_1$.\\
 	Choose $M>2\max \{\|u\|_{\Omega}, \|u_1\|_{\Omega_1}, \|v\|_{\Omega}, \|v_1\|_{\Omega_1}\}$.\\
 	By Lemma \ref{lem.compacontinuous}, if $j\geq k$, we have,
 	\begin{center}
 		$\int\limits_{\{u_k<v_j\}}(dd^c v_j)^q\wedge T\leq\int\limits_{\{u_k<v_j\}}(dd^c u_k)^q\wedge T.$
 	\end{center}
 	Take $\epsilon>0$. Let $G_{\epsilon}$ be an open subset of $\Omega_1$ such that $C(G_{\epsilon}, \Omega_1, T)<\epsilon$ and 
 	the restrictions of $u$ and $v$ to $\Omega_1\setminus G_{\epsilon}$ are cotinuous. 
 	By Tietze's extension theorem, there is a function $f\in C(\Omega_1)$ such that $f=v$ on $\Omega_1\setminus G_{\epsilon}$.
 	By Lemma \ref{lem.weakconvergence}, we have
 	\begin{center}
 		$\int\limits_{\{u_k<f\}}(dd^c v)^q\wedge T\leq\liminf\limits_{j\to\infty}\int\limits_{\{u_k<f\}}(dd^c v_j)^q\wedge T.$
 	\end{center}
 	Moreover,
 	\begin{center}
 		$\{u_k<f\}\cup G_{\epsilon}=\{u_k<v\}\cup G_{\epsilon}.$
 	\end{center}
 	Therefore
 	\begin{flushleft}
 		$\begin{array}{ll}
 		\int\limits_{\{u_k<v\}}(dd^c v)^q\wedge T
 		&\leq \int\limits_{\{u_k<f\}}(dd^c v)^q\wedge T+\int\limits_{G_{\epsilon}}(dd^c v)^q\wedge T\\
 		&\leq\liminf\limits_{j\to\infty} \int\limits_{\{u_k<f\}}(dd^c v_j)^q\wedge T+\int\limits_{G_{\epsilon}}(dd^c v)^q\wedge T\\
 		&\leq\liminf\limits_{j\to\infty} \int\limits_{\{u_k<v_j\}\cup G_{\epsilon}}(dd^c v_j)^q\wedge T
 		+\int\limits_{G_{\epsilon}}(dd^c v)^q\wedge T\\
 		&\leq\liminf\limits_{j\to\infty} \int\limits_{\{u_k<v_j\}}(dd^c v_j)^q\wedge T+2M^nC(G_{\epsilon}, \Omega_1, T)\\
 		&\leq\lim\limits_{j\to\infty} \int\limits_{\{u_k<v_j\}}(dd^c u_k)^q\wedge T+2M^nC(G_{\epsilon}, \Omega_1, T)\\
 		&\leq \int\limits_{\{u_k\leq v\}}(dd^c u_k)^q\wedge T+2M^nC(G_{\epsilon}, \Omega_1, T).
 		\end{array}$
 	\end{flushleft}
 	Hence
 	\begin{equation}\label{1.proofcompa.eq}
 	\int\limits_{\{u_k<v\}}(dd^c v)^q\wedge T\leq \int\limits_{\{u_k\leq v\}}(dd^c u_k)^q\wedge T+2M^nC(G_{\epsilon}, \Omega_1, T).
 	\end{equation}
 	Fix $k_0>0$ and denote $C(k_0)=C(\{u_{k_0}>u+\delta\}\cap\Omega_2, \Omega_1, T)$. We have $\lim\limits_{k_0\to\infty}C(k_0)=0$.\\
 	Moreover, for any $k>k_0$,
 	\begin{center}
 		$\{u_k\leq v\}\subset \{u_{k_0}\leq v+\delta\}\cup (\{u_{k_0}>u+\delta\}\cap\Omega_2).$
 	\end{center}
 	Then
 	\begin{equation}\label{2.proofcompa.eq}
 	\int\limits_{\{u_k\leq v\}}(dd^c u_k)^q\wedge T\leq \int\limits_{\{u_{k_0}\leq v+\delta\}}(dd^c u_k)^q\wedge T+M^nC(k_0).
 	\end{equation}
 	By the compactness of $\{u_{k_0}\leq v+\delta\}$, we have
 	\begin{equation}\label{3.proofcompa.eq}
 	\limsup\limits_{k\to\infty}\int\limits_{\{u_{k_0}\leq v+\delta\}}(dd^c u_k)^q\wedge T
 	\leq \int\limits_{\{u_{k_0}\leq v+\delta\}}(dd^c u)^q\wedge T. 
 	\end{equation}
 	Combining \eqref{1.proofcompa.eq}, \eqref{2.proofcompa.eq} and \eqref{3.proofcompa.eq}, and letting $k\rightarrow\infty$, we get
 	\begin{center}
 		$\int\limits_{\{u<v\}}(dd^c v)^q\wedge T\leq \int\limits_{\{u_{k_0}\leq v+\delta\}}(dd^c u)^q\wedge T+2M^nC(G_{\epsilon}, \Omega_1, T)
 		+M^nC(k_0).$
 	\end{center}
 	Letting $k_0\rightarrow\infty$ and $\epsilon\rightarrow 0$, we obtain
 	\begin{center}
 		$\int\limits_{\{u<v\}}(dd^c v)^q\wedge T\leq\int\limits_{\{u<v+\delta\}}(dd^c u)^q\wedge T.$
 	\end{center}
 \end{proof}


\begin{thebibliography}{BB}
\bibitem[BT82]{BT82}{E. BEDFORD, B. A. TAYLOR: \it A new capacity for plurisubharmonic functions.
 \rm Acta Math. \bfseries{149} \rm (1982), no. 1-2, 1--40.}
\bibitem[Bed93]{Bed93}{E. BEDFORD: \it Survey of pluri-potential theory,
	  \rm Several complex variables (Stockholm, 1987/1988), 48--97, Math. Notes, \bf{38},
	   \rm Princeton Univ. Press, Princeton, NJ, 1993.}
\bibitem[Blo04]{Blo04}{Z. BLOCKI: \it On the definition of the Monge-Amp\`ere operator in $\C^2$. 
	\rm Math. Ann. \bf{328} \rm (2004), no.3, 415--423.}
\bibitem[Blo06]{Blo06}{Z. BLOCKI: \it The domain of definition of the complex Monge-Amp`ere operator.
	 \rm Amer. J. Math.	\bf{128} \rm (2006), no.2, 519--530.}
\bibitem[Blo09]{Blo09}{Z. BLOCKI: \it Remark on the definition of the complex Monge-Ampère operator.
	\rm Functional analysis and complex analysis, 17--21, Contemp. Math., \bf{481}, \rm Amer. Math. Soc., Providence, RI, 2009.}
\bibitem[Ceg86]{Ceg86}{U. CEGRELL: \it Sums of continuous plurisubharmonic functions and the complex Monge-Amp\`ere operator in $\C^n$.
	\rm Math. Z. \bfseries{193} \rm (1986),  373--380.}
\bibitem[Ceg88]{Ceg88}{U. CEGRELL: \it Capacities in complex analysis, \rm  Aspects of Mathematics.
Vieweg, Wiesbaden (1988).}
\bibitem[Ceg04]{Ceg04}{U. CEGRELL: \it The general definition of the complex Monge-Amp\`ere operator.
	 \rm (English, French summary) Ann. Inst. Fourier (Grenoble) \bfseries{54} \rm (2004), no. 1, 159--179.}
\bibitem[Dem]{Dem}{J.-P. DEMAILLY: \it Complex analytic and differential geometry. \rm Free accessible book (http://www-fourier.ujf-grenoble.fr/~demailly/documents.html).}
\bibitem[Kli91]{Kli91}{M. KLIMEK: \it Pluripotential theory, \rm Oxford Univ. Press, Oxford, 1991.}
\bibitem[Kol05]{Kol05}{S. KOLODZIEJ: \it The complex Monge-Amp\`ere equation and pluripotential theory.
	\rm  Mem. Amer. Math. Soc. \bfseries{178} \rm (2005), no. 840.}
\bibitem[NP09]{NP09}{V.K. NGUYEN, H.H. PHAM: \it A comparison principle for the complex Monge-Amp\`ere operator in Cegrell's classes and applications. \rm Trans. Amer. Math. Soc. \bfseries{361} \rm (2009), no. 10, 5539--5554.}
\bibitem[Ras98]{Ras98}{A. RASHKOVSKII: \it Maximal plurisubharmonic functions associated with holomorphic mappings. 
	\rm Indiana Univ. Math. J. \bfseries{47} \rm (1998), no. 1, 297--309.}
\bibitem[Sad81]{Sad81}{A. SADULLAEV: \it Plurisubharmonic measures and capacities on complex manifolds.
	\rm Russian	Math. Surv. \bf{36} \rm (1981), 61--119.}
\bibitem[Xing96]{Xing96}{Y. XING: \it Continuity of the complex Monge-Amp\`ere operator. 
	\rm Proc. Amer. Math. Soc., \bfseries{124} (1996), 457--467. }
\bibitem[Wal68]{Wal68}{J. B. WALSH: \it Continuity of envelopes of plurisubharmonic functions. \rm J. Math. Mech. 
	\bfseries{18} \rm (1968), 143--148.}
\end{thebibliography}
\end{document}